\documentclass[11pt]{amsart}
\usepackage{amsmath}
\usepackage{amsfonts}
\usepackage{amssymb}

\usepackage{hyperref}
\hypersetup{colorlinks=true,
	linkcolor=blue,
	filecolor=magenta,
	citecolor=red,
	urlcolor=cyan,}
\newtheorem{theorem}{Theorem}
\theoremstyle{plain}

\numberwithin{equation}{section}
\numberwithin{lemma}{section}
\numberwithin{theorem}{section}
\numberwithin{corollary}{section}
\numberwithin{proposition}{section}
\numberwithin{remark}{section}

\begin{document}
	\title[The structure of $U(F(C_3 \times D_{10}))$]{The structure of the unit group of the group algebra $F(C_3 \times D_{10})$}
	\author{Meena Sahai}
	\address{Department of Mathematics and Astronomy, University of Lucknow, Lucknow, U.P. 226007, India.}
	\email{meena\_sahai@hotmail.com}
	\author{Sheere Farhat Ansari}
	\address{Department of Mathematics and Astronomy, University of Lucknow, Lucknow, U.P. 226007, India.}
\email{sheere\_farhat@rediffmail.com}
	
	\footnote { The financial assistance provided  to the second author in the form of a Senior Research Fellowship from the University Grants Commission, INDIA is gratefully acknowledged.}
	
	\maketitle

	\begin{abstract} 
		Let $D_{n}$ be the dihedral group of order $n$.	The structure of the unit group $U(F(C_3 \times D_{10}))$ of the group algebra $F(C_3 \times D_{10})$ over a finite field $F$ of characteristic $3$ is given in \cite{sh13}. In this article, the structure of $U(F(C_3 \times D_{10}))$ is obtained over any finite field $F$ of characteristic $p \neq 3$.
	\end{abstract}
	
	\subjclass{Mathematics Subject Classification 2010: 16U60; 20C05.}

	\keywords{Keywords: Group Rings, Unit Groups, Dihedral Groups, Cyclic Groups}
	\section{\textbf{Introduction}}
	Let  $U(FG)$ be the group of invertible elements of the group algebra $FG$ of a group $G$ over a field $F$. The study of units and their properties is one of the most challenging problems in the theory of group rings.  Explicit calculations in $U(FG)$ are usually difficult, even when $G$ is fairly small and $F$ is a finite field.  The results obtained in this direction are also useful for the investigation of the Lie properties of group rings, the isomorphism problem and other open questions in this area, see \cite{BK1}.
	
Let $J(FG)$ be the Jacobson radical of $FG$ and let $V=1+J(FG)$. The $F$-algebra $FG/J(FG)$ is semisimple whenever $G$ is a finite group.	It is known from the Wedderburn structure theorem that $$FG/J(FG) \cong \oplus_{i=1}^r M(n_i, K_i)$$ where $r$ is the number of non-isomorphic irreducible $FG$ modules, $n_i \in \mathbb{N}$ and $K_i$'s are finite dimensional division algebras over $F$. In this context a result by Ferraz \cite[Theorem 1.3 and Prop 1.2]{F1}) is very useful in determining the  Wedderburn decomposition of $FG/J(FG)$.

	If $FG$ is semisimple, then  $J(FG)=0$ and  by \cite[Prop 3.6.11]{Miles}, $$FG \cong F(G/G') \oplus \Delta(G, G')$$ where $F(G/G')$ is the sum of all the commutative simple components of $FG$, whereas $\Delta(G, G')$ is the sum of all the non-commutative simple components of $FG$.	We conclude that, if $FG$ is semisimple, then $$FG \cong F(G/G') \oplus (\oplus_{i=1}^lM(n_i, F_i)).$$
	Now, if $dim_F(Z(FG))=r$ and if the number of commutative simple components is $s$, then $l \leq r-s$.
	
		Many authors  \cite{sh10, J5, J12, N8, FM, sh4, sh13, T2} have studied the structure of $U(FG)$ for a finite group $G$ and for a finite field $F$. The structure of  $U(F(C_3 \times D_{10}))$ for $p=3$ is given in \cite{sh13}. 
	In this article, we provide an explicit description for the Wedderburn decomposition of $FG/J(FG)$, $G=C_3 \times D_{10}$ and $F$ a finite field of characteristic $p \neq 3$, using the theory developed by Ferraz~\cite{F1}. With the help of this description we obtain the structure of $U(F(C_3 \times D_{10}))$.

	 Our notations are same as in \cite{sh10, sh4}.
\section{\textbf{Structure of $U(F(C_3 \times D_{10}))$}}
\begin{theorem}
	Let $F$ be a finite field of characteristic $p$ with $|F|=q=p^k$ and let $G=C_3 \times D_{10}$. 
	\begin{enumerate}	
		\item If $p=2$, then $U(FG) \cong$
		
$\begin{cases}
	C_{2}^{3k} \rtimes\big(	C_{2^k-1}^{3} \times GL(2,F)^6\big), & \text {if $q \equiv  1, 4$  mod $15$;}\\
		%C_{2}^{3k} \rtimes	\big(C_{2^k-1}^3 \times GL(2, F_2)^3\big), & \parbox[t]{.6\textwidth}{ if $q \equiv -1, -4, -2$,\\~ $7$  mod $15$;}\\
		C_{2}^{3k} \rtimes	\big(C_{2^k-1} \times C_{2^{2k}-1} \times GL(2, F_2)^3\big), & \text{ if $q \equiv  2, -7$  mod $15$.}\\
	\end{cases}$
	
		\item If $p=5$, then
		
	$U(FG) \cong V \rtimes
	\begin{cases}
	C_{5^k-1}^6, & \text{if $q \equiv 1$ mod $6$};\\
	C_{5^k-1}^2 \times C_{5^{2k}-1}^2, & \text{if $q \equiv -1$ mod $6$}.
	\end{cases}$
	
	where
	$V \cong (C_5^{15k} \rtimes C_5^{6k}) \rtimes C_5^{3k}$ and  $Z(V) \cong C_5^{9k}$.
	
\item	If $p>5$, then $U(FG) \cong$

$\begin{cases}
	C_{p^k-1}^{6} \times GL(2,F)^6, & \text {if $q \equiv  1, -11$  mod $30$;}\\
	C_{p^k-1}^2 \times C_{p^{2k}-1}^2 \times GL(2, F)^2 \times GL(2, F_2)^2, & \text{ if $q \equiv -1, 11$  mod $30$;}\\
	C_{p^k-1}^6 \times GL(2, F_2)^3, & \text{ if $q \equiv  7, 13$  mod $30$;}\\
	C_{p^k-1}^2 \times C_{p^{2k}-1}^2 \times GL(2, F_2)^3, & \text{ if $q \equiv  -7, -13$  mod $30$.}\\
	\end{cases}$
	\end{enumerate}
\end{theorem}
\begin{proof}Let $G= \langle x, y, z \mid x^{2}=y^5=z^3=1,  xy=y^4x,  xz=zx, yz=zy \rangle.$ The conjugacy classes in $G$ are:
	\begin{align*}
	[z^i]&=\{z^i\} \text { for } i=0, 1, 2;\\
	[yz^i]&=\{y^{\pm 1}z^i\} \text { for } i=0, 1, 2;\\
	[y^2z^i]&=\{y^{\pm 2}z^i\} \text { for } i=0, 1, 2;\\
	[xz^i]&=\{xz^i, xy^{\pm 1}z^i, xy^{\pm 2}z^i\} \text { for } i=0, 1, 2.
	\end{align*}
	\begin{enumerate}
\item $p=2$. Clearly, $\widehat{T_2}=1+x\widehat{y}$.

Let $\alpha= \sum_{k=0}^{1}\sum_{j=0}^{2}\sum_{i=5(j+3k)}^{5(j+3k)+4}a_{i}x^ky^{i-5(j+3k)}z^j$. If $\alpha\widehat{T_2}=0$, then we have
$$\alpha +\sum_{k=0}^{1}\sum_{j=0}^{2}\sum_{i=5(j+3k)}^{5(j+3k)+4}a_ix^{k+1} \widehat{y}z^j =0.$$
		
		For $k=0, 1, 2$ and $i=0, 1, 2, 3, 4$ this yields the following equations:
		\begin{align*}
		a_{5k+i}+\sum_{j=0}^{4}a_{5k+j+15}&=0\\
		a_{5k+15+i}+\sum_{j=0}^{4}a_{5k+j}&=0.
		\end{align*}
		%\begin{align*}
		After simplification we get, $a_{5k}=a_{5k+i}=a_{5k+i+15}$ for $i=0, 1, 2, 3, 4$ and $k=0,1, 2$. Hence
$$Ann(\widehat{T_2})=\{\sum_{i=0}^{2}\beta_{i}(1+x)\widehat{y}z^i \mid  \beta_i\in F\}.$$
		%\end{center}
	  Since $z, \widehat{y} \in Z(FG)$,  $Ann^2(\widehat{T_2})=0$ and  $Ann(\widehat{T_2}) \subseteq J(FG)$. Thus by \cite[Lemma 2.2]{T2}, $J(FG) = Ann(\widehat{T_2})$ and  $dim_F(J(FG))=3$. Hence $V\cong C_2^{3k}$
	  and by  \cite[Lemma 2.1]{N2}, 
	  %\begin{center}
	  	$$U(FG) \cong C_2^{3k} \rtimes U(FG/J(FG)).$$
	 % \end{center}
	Now it only remains to find  the Wedderburn decomposition of $FG/J(FG)$.
	
		As $[1]$, $[y]$, $[y^2]$, $[z]$, $[z^2]$, $[yz]$, $[yz^2]$, $[y^2z]$, and $[y^2z^2]$ are the $2$-regular conjugacy classes of $G$,  $m=15$ and $dim_F(FG/J(FG))=27$. Now the following cases occur:
		\begin{enumerate}
	\item If $q \equiv 1, 4$ mod $15$, then $|S_F(\gamma_g)|=1$ for $g=1$, $y$, $y^2$, $z$, $z^2$, $yz$, $yz^2$, $y^2z$, $y^2z^2$. Consquently, \cite[Theorem 1.3]{F1},  yields nine components in the  decomposition of $FG/J(FG)$.  In view of the dimension requirements, the only possibility is:
$$FG/J(FG) \cong F^3 \oplus M(2,F)^6.$$
			
\item If $q \equiv 2, -7$ mod $15$, then 
		$|S_F(\gamma_g)|=1$ for $g=1$ and $|S_F(\gamma_g)|=2$ for $g=y$, $z$, $yz$, $yz^2$. So, due to the dimension restrictions, we have 
		$$FG/J(FG) \cong F \oplus F_{2} \oplus M(2,F_2)^3.$$
		%\end{center}
	\end{enumerate}
	
\item $p=5$. If $K=\langle y \rangle$, then $G/K \cong H \cong \langle x, z\rangle \cong C_6$. Thus 
  from the ring epimorphism $\eta : FG \rightarrow FH$, given by 

  	\begin{align*}
\eta (\sum_{j=0}^{2}\sum_{i=0}^{4}y^{i}z^j(a_{i+5j} + a_{i+5j+15}x))= \sum_{j=0}^{2}\sum_{i=0}^{4}z^j(a_{i+5j}+ a_{i+5j+15}x),
\end{align*}
  we get a group epimorphism  $ \phi \colon U(FG) \rightarrow U(FH)$ and $ker \phi \cong 1+J(FG) = V$. Further, we have the inclusion map 
%$ FH \rightarrow FG$, 
  %$$\mu (\sum_{i=0}^{2}z^{i}(b_i+b_{i+3}x) ) =  \sum_{i=0}^{2}z^{i}(b_i+b_{i+3}x).$$
%  we get a group monomorphism 
$i \colon U(FH) \rightarrow U(FG)$ such that $\phi i =1_{U(FH)}$. Thus $U(FG) \cong V \rtimes U(FC_6)$.
  
  The structure of $U(FC_6)$ is given in \cite[Theorem 4.1]{sh4}.
  
If $v =\sum_{j=0}^2\sum_{i=0}^{4}y^iz^j(a_{i+5j}+a_{i+5j+15}x) \in U(FG)$, then $v \in V$ if and only if $\sum_{i=0}^{4}a_i=1$ and  $\sum_{i=0}^{4}a_{i+5k}=0$ for $k=1, 2, 3, 4, 5$.
Hence
%$$V=\{1+\sum_{i=1}^{4}(y^i-1)(b_i + b_{i+4}z + b_{i+8}z^2 + b_{i+12}x+b_{i+16}xz+b_{i+20}xz^2) \mid b_i \in F\}$$
\begin{center}
$V=\{1+\sum_{j=0}^2\sum_{i=1}^{4}(y^i-1)z^j( b_{i+4j}+ b_{i+4j+12}x)\mid b_i \in F\}$
\end{center}
and $|V|=5^{24k}$. Since, $J(FG)^5=0$,  $V^5=1$.

Now we show that $V\cong (C_{5}^{15k} \rtimes C_5^{6k} )\rtimes C_5^{3k}$. The proof is split into the following steps:

{\bf Step 1:} Let $R=\{1+ ay(1-y)^3x \mid a \in F\} \subseteq V$. Then $R\cong C_5^{k}$.

If $$r_1=1+ay(1-y)^3x\in R$$ and $$r_2=1+by(1-y)^3x\in R$$ where $a, b\in F$, then $$r_1r_2=1+(a+b)y(1-y)^3x\in R.$$ Therefore, $R$ is an abelian subgroup of $V$ of order $5^k$. Hence $R\cong C_5^k$.

{\bf Step 2:} $|C_V(R)|=5^{21k}$, where $C_V(R)=\{v \in V \mid r^v=r \text { for all } r \in R \}$.

Let $$r=1+ay(1-y)^3x \in R$$ and $$v=1+\sum_{j=0}^{2} \sum_{i=1}^{4}(y^i-1)z^j(b_{i+4j} +b_{i+4j+12}x) \in V$$ where $a, b_i \in F$. Then
$v = 1+  v_1 + v_2x$, $v_1 = \sum_{j=0}^{2}\sum_{i=1}^{4}b_{i+4j}(y^i-1)z^j$ and $v_2 =  \sum_{j=0}^{2}\sum_{i=1}^{4}b_{i+4j+12}(y^i-1)z^j$. So $v^{-1 } = v^4=1+4v_1+4v_2x$ mod $(y-1)^2FG$. Thus
\begin{align*}
r^v=1+ v^{-1}ay(1-y)^3xv = r+2a\widehat{y}\sum_{j=0}^2\sum_{i=1}^{4}ib_{i+4j}z^jx.
\end{align*}
Thus $r^v=r$ if and only if $\sum_{i=1}^{4}ib_{i+4j}=0$ for $j=0, 1, 2$. Hence
\begin{align*}
%N_V(S)=\{1+\sum_{i=1}^{3}(y^i-1)(c_i+c_{i+3}z+c_{i+6}z^2)+(y^4-1)\sum_{i=1}^{3}i(c_i+c_{i+3}z+c_{i+6}z^2)+\\\sum_{i=1}^{4}(y^i-1)(c_{i+9}+c_{i+13}z+c_{i+17}z^2)x \mid c_i \in F\}
C_V(R)=\{1+\sum_{j=0}^2\sum_{i=1}^{3}[(y^i-1)+i(y^4-1)]c_{i+3j}z^j&\\  +\sum_{j=0}^2\sum_{i=1}^{4}(y^i-1)c_{i+4j+9}z^jx \mid c_i \in F\}
\end{align*}
and $|C_V(R)|=5^{21k}$.

{\bf Step 3:} $C_V(R)\cong C_{5}^{15k}\rtimes C_5^{6k}$.

Consider the sets $$S=\{1 + y^3(y-1)^2[yb_1+y(y+2)b_2+b_3+(yb_4+(y+1)^2b_5)x]\}$$
and
%where $b_{1+j}=\sum_{i=0}^2p_{i+3j}z^i$ for $j=0, 1, 2, 3, 4$. It can easily be shown that $S$ is an abelian group and  $S \cong C_{5}^{15k}$. Also it can be shown that the set
 $$T=\{1+y^3(y-1)[(y-1)(yc_1+(y+1)^2c_2) + (yc_3 + (y^2 +y+1)c_4)x]\}$$
where $b_{1+j}=\sum_{i=0}^2p_{i+3j}z^i$ for $j=0, 1, 2, 3, 4$ and $c_{1+j}=\sum_{i=0}^2q_{i+3j}z^i$ for $j=0, 1, 2, 3$ . With some computation it can be shown that $S$ and $T$ are abelian subgroups of $C_V(R)$. So $S \cong C_{5}^{15k}$ and $T \cong C_{5}^{12k}$.

Now, let 
$$s=1 + y^3(y-1)^2[yb_1+y(y+2)b_2+b_3+(yb_4+(y+1)^2b_5)x] \in S$$
 and
 $$t=1+y^3(y-1)[(y-1)(yc_1+(y+1)^2c_2)+(yc_3 + (y^2 +y+1)c_4)x] \in T.$$ Then
\begin{align*}
 s^t=&1+y^3(y-1)^2\{yb_1+y(y+2)b_2+b_3+k_1y^3(1-y)\\
 &+[yb_4+(y+1)^2b_5+(y-1)^2(k_2+k_3)]x\} \in S
 \end{align*}
where
\begin{align*}
k_1&=(c_4+2c_3)(b_4-b_5),  k_2=(c_4+2c_3)(b_2-b_3)\\
 k_3&=2(c_4^2-c_3c_4-c_3^2)(b_4-b_5).
 \end{align*}
  Let $$U=S \cap T=\{1+y^3(y-1)^2[yc_1+(y+1)^2c_2]\}$$
  where $c_{1+j}=\sum_{i=0}^{2}q_{i+3j}z^i$ for $j=0, 1$. Thus $U \cong C_5^{6k}$.
   So for some subgroup $W \cong C_5^{6k}$ of $T$, $T=U \times W$ and $W\cap S=1$. Hence $C_V(R) \cong S\rtimes W \cong C_{5}^{15k} \rtimes C_5^{6k}$.
   
   {\bf Step 4:} Let $M=\{1+\sum_{j=0}^{2}r_jz^jy(y+1)^2(1-y)(1+x) \mid r_i \in F\} \subseteq V$. Then $M\cong C_5^{3k}$.
   
   Let $$m_1=1+\sum_{j=0}^{2}r_jz^jy(y+1)^2(1-y)(1+x)\in M$$ and $$m_2=1+\sum_{j=0}^{2}s_jz^jy(y+1)^2(1-y)(1+x)\in M$$ where $r_j, s_j\in F$. Then
   \begin{align*}
    m_1m_2=1+\sum_{j=0}^{2}(r_j +s_j)z^jy(y+1)^2(1-y)(1+x)\in M.
    \end{align*}
     Therefore, $M$ is an abelian subgroup of $V$ of order $5^{3k}$. Hence, $M\cong C_5^{3k}$.
   
   {\bf Step 5:} $V\cong C_V(R) \rtimes M$.
   
 Let
   \begin{align*}
    n=&1+\sum_{j=0}^2\sum_{i=1}^{3}[(y^i-1)+i(y^4-1)]c_{i+3j}z^j\\
 & +\sum_{j=0}^2\sum_{i=1}^{4}(y^i-1)c_{i+4j+9}z^jx  \in C_V(R)
    \end{align*}
    and let $$m=1+\sum_{j=0}^{2}r_jz^jy(y+1)^2(1-y)(1+x)\in M$$ 
    where $c_i, r_i\in F$. Then 
    \begin{align*}
    n^m=&1+\sum_{j=0}^2\sum_{i=1}^{3}[(y^i-1)+i(y^4-1)]c_{i+3j}z^j\\
&+\sum_{j=0}^2\sum_{i=1}^{4}(y^i-1)c_{i+4j+9}z^jx +(k_1+k_2x)\in C_V(R)
    \end{align*}
    where
    \begin{align*}
     k_1=&\sum_{j=0}^{2}r_jz^j\{\sum_{k=0}^{2}(c_{10+4k}-c_{11+4k}-c_{12+4k}+c_{13+4k})z^k\\
&+3\sum_{j=0}^{2}r_jz^j\sum_{i=1}^{4}\sum_{k=0}^{2}i(c_{i+4k+9}z^k)\}y(1-y)^3
     \end{align*}
   
   \begin{align*} 
    k_2=&2\sum_{j=0}^{2}r_jz^j\{\sum_{k=0}^{2}(c_{2+3k}-c_{3+3k})z^k(1-y)\\
&-\sum_{j=0}^{2}r_jz^j\sum_{i=1}^{4}\sum_{k=0}^{2}ic_{i+4k+9}z^k\}y(1-y)^3-2\sum_{j=0}^{2}r_jz^j\sum_{i=0}^{4}d_iy^i
    \end{align*}
    where 
    \begin{align*}
    d_0&=\sum_{j=0}^{2}(4c_{10+4j}+4c_{11+4j}+c_{12+4j}+c_{13+4j})z^j,\\
     d_1&=\sum_{j=0}^{2}(4c_{10+4j}+3c_{11+4j}+3c_{12+4j})z^j,\\
     d_2&=\sum_{j=0}^{2}(4c_{11+4j}+3c_{12+4j}+3c_{13+4j})z^j,\\
     d_3&=\sum_{j=0}^{2}(2c_{10+4j}+2c_{11+4j}+c_{12+4j})z^j,\\
     d_4&=\sum_{j=0}^{2}(2c_{11+4j}+2c_{12+4j}+c_{13+4j})z^j.
    \end{align*}
     Clearly, $C_V(R) \cap M=1$. Therefore, $V=C_V(R) \rtimes M$.
    
    In the sequel, we show that $Z(V) \cong C_5^{9k}$. 
    %Put $$C_V(y) = \{v \in V \mid vy=yv\}$$
    
     If $v=1+\sum_{j=0}^2\sum_{i=1}^{4}(y^i-1)z^j(b_{i+4j}+b_{i+4j+12}x) \in C_V(y)=\{v \in V \mid vy=yv\}$, then  
    %\begin{align*}
    $$vy-yv=\sum_{i=1}^{4}\sum_{j=0}^{2}y(1-y^i)(y^3-1)b_{i+4j+12}z^jx.$$
    %\end{align*}
    
    Thus $v \in C_V(y)$ if and only if  $b_{i}=b_{i+j}$ for $j= 1, 2, 3$ and $i=13, 17, 21$.
    Hence
   $$ C_V(y) =\{1+\sum_{j=0}^2\sum_{i=1}^{4}(y^i-1)c_{i+4j}z^j+\widehat{y}\sum_{j=0}^2c_{j+13}z^jx \mid c_i \in F\}.$$
   % \end{align*}  
    
    Since $Z(V) \subseteq C_V(y)$, 
    $$Z(V) =\{s\in C_V(y) \mid sv=vs \text{ for all } v \in V\}.$$
    Let $u=1+\sum_{j=0}^2\sum_{i=1}^{4}(y^i-1)c_{i+4j}z^j+\widehat{y}x\sum_{j=0}^2c_{j+13}z^j\in C_V(y)$. Since $v = 1+(y-1)zx \in V$ and $\widehat{y} \in Z(FG)$, $vu-uv=0$ yields
    $$(y-1)\sum_{i=1}^{4}\sum_{j=0}^{2}(y^i-y^{-i})c_{i+4j}z^{j+1}x = 0.$$
    Thus $c_i=c_{i+3}$ for $i=1, 5, 9$ and $c_{j}=c_{j+1}$ for $j=2, 6, 10$ and $u= 1+y^4(y-1)^2\sum_{j=0}^{2}d_{1+j}z^j+y^3(y^2-1)^2\sum_{j=0}^{2}d_{4+j}z^j+\widehat{y}\sum_{j=0}^{2}d_{7+j}z^jx$. Clearly $u \in Z(V)$.
   
    We conclude that $Z(V)=\{1+y^4(y-1)^2\sum_{j=0}^{2}d_{1+j}z^j+y^3(y^2-1)^2\sum_{j=0}^{2}d_{4+j}z^j+\widehat{y}\sum_{j=0}^{2}d_{7+j}z^jx \mid d_i \in F \} \cong C_5^{9k}$.
    	
\item If $p>5$, then $J(FG)=0$. Thus $FG$ is semisimple and $m=30$. As $G/G' \cong C_6$, we have 
	%\begin{equation}\label{}
	$$FG \cong FC_{6} \oplus \big (\oplus_{i=1}^l M(n_i, F_i)\big).$$
	%\end{equation}
	Since $dim_F(Z(FG))=12$,  $l \leq 6$.
	Now we have the following cases:
	\begin{enumerate}
\item	If $q\equiv 1, -11$ mod $30$, then  $|S_F(\gamma_g)|=1$ for all $g\in G$. Therefore by  \cite[Theorem 4.1]{sh4} and \cite[Prop 1.2 and Theorem 1.3]{F1}, 
	$$FG \cong F^{6} \oplus \big(\oplus_{i=1}^6 M(n_i, F) \big)$$
	and $\sum_{i=1}^{6}n_i^2=24$.
	Clearly $n_i=2$ for $i\in \{1, 2, 3, 4, 5, 6\}$.
	Hence,
	\begin{center}
		$FG\cong F^{6} \oplus M(2,F)^6$.
	\end{center}
	
\item	If $q \equiv -1, 11$ mod $30$, then
	$|S_F(\gamma_g)|=1$ for $g=1, x, y, y^2$ and $|S_F(\gamma_g)|=2$ for $g=z, xz, yz, y^2z$. In this case $FC_6 \cong F^2 \oplus F^2_2$,  thus dimension constraints yield
	
	\begin{center}
		$n_1^2+n_2^2+2n_3^2+2n_4^2=24$.
	\end{center}
	We get $n_1=n_2=n_3=n_4=2$. Hence,
	\begin{center}
		$FG \cong F^2 \oplus F_2^2 \oplus M(2,F)^2 \oplus M(2, F_2)^2$.
	\end{center}
	
\item	If $q \equiv 7, 13$ mod $30$, then
	$T=\{1, 7, 13,  19\}$ mod $30$. Thus 
		$|S_F(\gamma_g)|=1$ for $g=1, x, z, z^2, xz, xz^2$ and $|S_F(\gamma_g)|=2$ for $g=y, yz, yz^2$.  Therefore,
	
	\begin{center}
		$2(n_1^2+n_2^2+n_3^2)=24$.
	\end{center}
	We get $n_1=n_2=n_3=2$. Hence,
	\begin{center}
		$FG \cong F^{6} \oplus M(2, F_2)^3$.
	\end{center}  
	
\item	If $q \equiv -7, -13$ mod $30$, then
	$T=\{1, 17, 19, 23\}$ mod $30$. Thus
		$|S_F(\gamma_g)|=1$ for $g=1, x$ and $|S_F(\gamma_g)|=2$ for $g=y, z,  xz, yz, yz^2$. 
	Hence,
	\begin{center}
		$FG \cong F^{2} \oplus F_2^2 \oplus M(2, F_2)^3$.
	\end{center}  
	
\end{enumerate}
\end{enumerate}
\end{proof}

\end{document}